\documentclass[12pt,leqno]{article}

\usepackage[cp1250]{inputenc}
\usepackage[OT4]{fontenc}
\usepackage{amsfonts}
\usepackage{amsmath}
\usepackage{amssymb}
\usepackage{fancyhdr}
\newcommand{\R}{\mathbb{R}}
\newcommand{\C}{\mathbb{C}}
\newcommand{\N}{\mathbb{N}}
\newcommand{\bO}{{\bf 0}}
\newcommand{\bm}{{\bf m}}
\newcommand{\Bpis}{{\cal B}}
\newcommand{\Opis}{{\cal O}}
\newcommand{\Npis}{{\cal N}}
\newcommand{\Mpis}{{\cal M}}
\newcommand{\K}{\mathbb{K}}
\newcommand{\rank}{\mathop{\rm rank}\nolimits}
\newcommand{\ar}{\longrightarrow}
\newcommand{\inv}{^{-1}}

\newtheorem{theorem}{Theorem}[section]
\newtheorem{lemma}[theorem]{Lemma}
\newtheorem{cor}[theorem]{Corollary}
\newtheorem{prop}[theorem] {Proposition}

\newenvironment{ex}{\medskip \noindent {\bf Example.\ }}{\bigskip}
\newenvironment{proof}{\par\noindent \emph{Proof. }}{\hspace*{\fill}$\Box$\par\medskip}

\title{On the number of branches of real curve singularities\thanks{%
Aleksandra~Nowel and Zbigniew~Szafraniec\\
University of Gda\'{n}sk,
              Institute of Mathematics \\
              80-952 Gda\'{n}sk, Wita Stwosza 57\\
              Poland\\
              Tel.: +48-58-5232059\\
              Fax: +48-58-3414914\\
              Email: Aleksandra.Nowel@mat.ug.edu.pl\\
              Email: Zbigniew.Szafraniec@mat.ug.edu.pl\\ \\
2000 \emph{Mathematics Subject Classification} 14H20 (primary), 32B10, 58K65 (secondary)}
}

\author{Aleksandra~Nowel \and Zbigniew~Szafraniec}

\date{June 2010}

\begin{document}

\def\nothanksmarks{\def\thanks##1{\protect\footnotetext[0]{\kern-\bibindent##1}}}

\nothanksmarks

\maketitle

\pagestyle{fancy}

\lhead{\fancyplain{}{\textsc{\small A.~Nowel, Z.~Szafraniec}}}
\rhead{\fancyplain{}{\emph{\small On the number of branches of real curve singularities}}}

%\markboth{\textsc{\small A.~Nowel, Z.~Szafraniec}}{\textsc{\small On the number of branches of real curve singularities}}

\begin{abstract}
There is presented a method for computing
the number of branches of a real analytic curve germ
$V(f_1,\ldots,f_m)\subset\R^n$ $(m\geq n)$ having a singular point
at the origin, and the number of half--branches of the set of double
points of an analytic germ $u:(\R^2,\bO)\rightarrow (\R^3,\bO)$.
\end{abstract}

\bigskip

Let $f_1,\ldots,f_m:\R^n,\bO\rightarrow \R,0$ be germs of real analytic functions,
and let $V=\{x\ | \ f_1(x)=\ldots=f_m(x)=0\}$
be the corresponding germ of an analytic set.
If $\dim V\leq 1$ then $V$ is locally a union of a finite collection of semianalytic half-branches.
Let $b_0$ denote the number of half-branches.

Several authors presented algebraic methods for computing $b_0$. The case
of planar curves has been investigated by Cucker {\em et al.} \cite{cuckeretal}.
In the case where $m=n-1$ and $V$ is a complete intersection
Fukuda {\em et al.} \cite{aokietal1}, \cite{aokietal2}, \cite{fukudaetal} have given a formula for $b_0$.
They have associated to $f_1,\ldots, f_{n-1}$ a map germ $\R^n,\bO\rightarrow \R^n,\bO$ whose 
topological degree equals $\frac{1}{2}b_0$. Then one may use the Eisenbud-Levine \cite{eisenbudlevine} 
and Khimshiashvili \cite{khimshiashvili1}, \cite{khimshiashvili2} theorem to calculate
the degree as the signature of a quadratic form on the corresponding local algebra.

This has been extended by Montaldi and van Straten \cite{montaldivanstraten}
to the more general case where the curve is not a complete intersection.
An analytic 1-form $\alpha$ defines an orientation on each half-branch.
They have shown that one may associate to $V$ and $\alpha$
so called "ramification modules" together with real valued non-degenerate
quadratic forms whose signatures determine the number of "outbound" and "inbound"
half-branches. In particular, if $\alpha=\sum x_i\, dx_i$ then every half-branch is "outbound",
and so $b_0$ can be expressed as the signature of the appropriate quadratic form.

Damon \cite{damon1}  has applied the method of Montaldi-van Straten so as to give a very effective
formula for $b_0$ in the case of a weighted homogeneous curve singularity. He has also
extended this result to the $G$-equivariant case \cite{damon2}. 

There exist efficient computer programs which may compute the local topological degree
using the Eisenbud-Levine and Khimshiashvili method
(see \cite{leckiszafraniec1}, \cite{leckiszafraniec3}).
On the other hand, often one may need to compute $b_0$ if $m>n-1$, so that $V$ is not
a complete intersection.

Assume that $m>n-1$ and the origin is isolated in the set of $z\in\C^n$ such that
$f_1(z)=0,\ldots, f_m(z)=0$ and the rank of the derivative $D(f_1,\ldots,f_m)(z)<n-1$.
In this paper we shall show how to construct two mappings $H_\pm:\R^n,\bO\rightarrow \R^n,\bO$
such that $b_0$ is the difference of the local degrees $\deg_0(H_+)$ and $\deg_0(H_-)$.

The paper is organized as follows. In Section 1 and Section 2 we have collected
some useful facts about ideals in the ring of convergent power series, and about
initial diagrams of these ideals. In Section 3 we prove that one may find germs
$g_1,\ldots,g_{n-1},h$ which are linear combination of $f_1,\ldots, f_m$ such that
$V=\{x\ | \ g_1(x)=\ldots=g_{n-1}(x)=h(x)=0\}$.
In Section 4 we apply results of \cite{szafraniec14} so as to prove the main result, 
and we compute some simple examples. 

Mappings $u:(\R^2,\bO)\rightarrow (\R^3,\bO)$ are a natural object of study in the theory of singularities
(see \cite{mond1}, \cite{mond2}, for a recent account we refer the reader to 
\cite{mararnunoballesteros}). In Section 5  we show how to verify that $u$ has only transverse double points
(which appear along a curve $D^2(u)$ in $\R^3$), and no triple points. 
We shall also show how to compute the number of half--branches in $D^2(u)$.

In this paper we present several examples computed by a computer. We have implemented our algorithm
with the help of {\sc Singular} \cite{singular}. We have also used a computer program written by
Andrzej {\L}\c{e}cki \cite{leckiszafraniec1}.

\section{Preliminaries}
Let $\K$ denote either $\R$ or $\C$.
For $p\in\K^n$, let $\Opis_{K,p}$ denote the ring of germs
at $p$ of analytic functions $\K^n\rightarrow\K$.
If $I_p$ is an ideal in $\Opis_{K,p}$, let $V_K(I_p)\subset\K^n$ denote
the germ of zeros of $I_p$ at $p$. Let $\bm_0=\{f\in\Opis_{K,0}\ |\ f(\bO)=0\}$ denote the maximal ideal in $\Opis_{K,0}$.

Let $U\subset\C^n$ be an open neighbourhood of the origin.
Suppose that functions $f_1,\ldots, f_s$ and $g_1,\ldots,g_t$ are holomorphic in $U$.
For $p\in U$, let $I_p=\left<f_1,\ldots,f_s\right>$ (resp. $J_p=\left<g_1,\ldots,g_t\right>$) denote the ideal in $\Opis_{C,p}$ generated by $f_1,\ldots,f_s$
(resp. by $g_1,\ldots , g_t$). 

\begin{prop}\label{nr1}
There exists $k$ such that $\bm_0^k I_0\subset J_0$ if and only if there is an open neighbourhood $W$ of the origin
such that $I_p\subset J_p$ for $p\in W\setminus\{\bO\}$.

If that is the case and $\bO\in V_C(I_0)$ then $V_C(J_0)\subset V_C(I_0)$.
\end{prop}
\begin{proof}  ($\Rightarrow$) Let $(z_1,\ldots,z_n)$ denote the coordinates in $\C^n$.
There exists an open neighbourhood $W$ of the origin such that each
function $z_j^kf_i$ is on $W$ a combination of $g_1,\ldots,g_t$ with holomorphic coefficients.
If $p=(p_1,\ldots,p_n)\in W\setminus \{\bO\}$ then some $p_j\neq 0$. As $z_j^k$ is invertible
in $\Opis_{C,p}$, so all $f_i\in J_p$, and then $I_p\subset J_p$.\\[1em]
($\Leftarrow$)  For $p\in U$ let
\[{\cal R}_i(p)=\{(a,h_1,\ldots,h_t)\in(\Opis_{C,p})^{t+1}\ |\ af_i=h_1g_1+\cdots +h_tg_t\}.\]
As $\Opis_{C,p}$ is a noetherian ring, ${\cal R}_i(p)$ is a finitely generated $\Opis_{C,p}$-module, and
there exists a finite family
$\{b_\ell=(a_\ell,h_{\ell 1},\ldots h_{\ell t})\}$ of generators of ${\cal R}_i(\bO)$.
Of course,
\[{\cal A}_i(p)=\{a\in\Opis_{C,p}\ |\ \exists\ h_1,\ldots,h_t\in\Opis_{C,p}\ :\ (a,h_1,\ldots,h_t)\in
{\cal R}_i(p)\}\]
is an ideal in $\Opis_{C,p}$, and $a\in{\cal A}_i(p)$ if and only if $af_i\in J_p$.

By the Oka coherence theorem, there exists an open $W$ such that
$\bO\in W\subset U$,
representatives of $\{b_\ell\}$ are defined and holomorphic in $W$, and for each $p\in W$
the germs of $\{b_\ell\}$ at $p$ generate ${\cal R}_i(p)$.
In particular, the germs of $\{a_\ell\}$ at $p$ generate ${\cal A}_i(p)$.
If $p\in W\setminus\{\bO\}$ lies sufficiently close
to the origin then $I_p\subset J_p$.
Hence $1\cdot f_i=f_i\in I_p\subset J_p$, and then $1\in {\cal A}_i(p)$.
Thus some $a_\ell(p)\neq 0$. Hence $ V_C(\{a_\ell\})\subset\{\bO\}$. 

The germs $\{a_\ell\}$ at $\bO$ generate ${\cal A}_i(\bO)$.
By R\"uckert's local Nullstellensatz, there exists $k(i)$ with
$\bm_0^{k(i)}\subset{\cal A}_i(\bO)$, and then
$\bm_0^{k(i)}\left< f_i\right>\subset J_0$.
Let $k=\max(k(1),\ldots,k(s))$. As $f_1,\ldots, f_s$
generate $I_0$, then $\bm_0^kI_0\subset J_0$.

If that is the case and $\bO\in V_C(I_0)$, then
$V_C(J_0)\subset V_C(\bm_0^k)\cup V_C(I_0)=\{\bO\}\cup V_C(I_0)=V_C(I_0)$. \end{proof}

\begin{cor}\label{nr2}
There exists $k$ such that $\bm_0^k I_0\subset J_0$ and $\bm_0^k J_0\subset I_0$ 
if and only if there is an open neighbourhood
$W$ of the origin such that $I_p=J_p$ for $p\in W\setminus\{\bO\}$.

If that is the case and $\bO\in V_C(I_0)\cap V_C(J_0)$ then $V_C(I_0)=V_C(J_0)$.
\end{cor}

\section{Diagrams of initial exponents}
In this section we present some properties of diagrams
of initial exponents.
In exposition and notation we follow closely  \cite{bierstonemilman}.

Let $\N$ denote the nonnegative integers.

If $\alpha=(\alpha_1,\ldots,\alpha_n)\in\N^n$, put
$|\alpha|=\alpha_1+\cdots+\alpha_n$. We order the
$(n+1)$--tuples $(\alpha_1,\ldots,\alpha_n,|\alpha|)$
lexicographically from the right. This induces a total
ordering of $\N^n$.

Let $0\neq f=\sum a_\alpha x^\alpha\in\Opis_{K,0}$, where $\alpha\in\N^n$,
$a_\alpha\in\K$, and
$x^\alpha=x_1^{\alpha_1}\cdots x_n^{\alpha_n}$. Denote
\[\operatorname{supp}(f)=\{\alpha\in\N^n\ |\ a_\alpha\neq 0\}.\]
Let $\nu(f)$ denote the smallest element of $\operatorname{supp}(f)$,
and let $\operatorname{in}(f)$ denote $a_{\nu(f)}x^{\nu(f)}$.
It is easy to verify that
$\nu(f_1\cdot f_2)=\nu(f_1)+\nu(f_2)$, 
$\operatorname{in}(f_1\cdot f_2)=\operatorname{in}(f_1)\operatorname{in}(f_2)$, and
$\nu(\sum f_i)=\min(\{\nu(f_i)\})$ if $\nu(f_i)$ are pairwise distinct.

We define the diagram of initial exponents $\Npis(I)$ of an ideal
$I\subset\Opis_{K,0}$ as
$\{\nu(f)\ |\ f\in I\setminus\{0\} \}$. Clearly, $\Npis(I)+\N^n=\Npis(I)$.
There is a smallest finite subset ${\cal B}(I)$ of $\Npis(I)$
such that $\Npis(I)={\cal B}(I)+\N^n$.
%We call ${\cal B}(I)$ the vertices of $\Npis(I)$.

If $I\subset J$ are ideals then $\Npis(I)\subset\Npis(J)$.
There exist $g^1,\ldots,g^s\in I$ and $g^{s+1},\ldots, g^t\in J$ such that
${\cal B}(I)=\{\nu(g^1),\ldots,\nu(g^s)\}$ and
${\cal B}(J)\setminus\Npis(I)=\{\nu(g^{s+1}),\ldots,\nu(g^t)\}$.
We associate to $g^1,\ldots,g^t$ the following decomposition of $\N^n$:
\[\Delta_1=\nu(g^1)+\N^n,\]
\[\Delta_i=(\nu(g^i)+\N^n)\setminus\bigcup_{k=1}^{i-1}\Delta_k,\ \ i=2,\ldots,t,\]
\[\Delta=\N^n\setminus\bigcup_{k=1}^t\Delta_k=\N^n\setminus\Npis(J).\]
As $\Npis(I)=\bigcup_{i=1}^s \Delta_i$ and $\Npis(J)=\bigcup_{i=1}^t\Delta_i$, then
$\Npis(J)\setminus\Npis(I)=\bigcup_{i=s+1}^t\Delta_i$.

\begin{theorem}[ \cite{arocaetal}, \cite{bierstonemilman}, \cite{grauert}]\label{nr3}
For every $f\in\Opis_{K,0}$ there exist unique $q_i\in\Opis_{K,0}$, $i=1,\ldots, t$ and
$r\in\Opis_{K,0}$ such that
\[\nu(g^i)+\operatorname{supp}(q_i)\subset\Delta_i,\ \ i=1,\ldots,t,\]
\[\operatorname{supp}(r)\subset\Delta,\]
and $f=\sum_{i=1}^t q_i g^i+r$.  
\end{theorem}

\begin{prop}\label{nr4}
$f\in I$ if and only if $q_{s+1}=\cdots =q_t=r=0$.
\end{prop}
\begin{proof}  ($\Leftarrow$) is obvious.\\[1em]
($\Rightarrow$) As $q_i$ and $r$ are unique, if $f=0$ then $r=0$ and all $q^i=0$.

Suppose, contrary to our claim, that $f\in I\setminus\{0\}$ and either $r\neq 0$ or $q_i\neq 0$
for some $s+1\leq i\leq t$. Put
\[h=\sum_{s+1}^tq_i g^i+r=f-\sum_1^s q_i g^i\ .\]

By Theorem \ref{nr3},
$\nu(q_i g^i)=\nu(g^i)+\nu(q_i)\in\Delta_i$,
and $\nu(r)\in\Delta$.
Thus $\operatorname{in}(h)$ is $\operatorname{in}(r)$ or one of the
$\operatorname{in}(q_i g^i)=\operatorname{in}(q_i)\operatorname{in}(g^i)$,
where $s+1\leq i\leq t$, since their exponents $\nu(r)$ and
$\nu(q_i g^i)$ lie in disjoint regions of $\N^n$. 
Hence $h\neq 0$, and
\[\nu(h)\in\Delta\cup \bigcup_{s+1}^t\Delta_i=\N^n\setminus\Npis(I).\]
On the other hand, $h=f-\sum_1^s q_i g^i\in I$ and then
$\nu(h)\in\Npis(I)$, a contradiction. \end{proof} 

Applying similar arguments one may prove

\begin{prop}\label{nr5}
$f\in J$ if and only if $r=0$. 
\end{prop}

\begin{prop}\label{nr6}
If $f\in J$ then its residue class in $J/I$ is uniquely represented by
$\sum_{s+1}^t q_i g^i$, where $\nu(g^i)+\operatorname{supp}(q_i)\subset\Delta_i$.
\end{prop}
\begin{proof}  By the above proposition, $r=0$, and then
\[f=\sum_{i=1}^t q_i g^i=\sum_{i=1}^s q_i g^i +\sum_{i=s+1}^t q_i g^i\ .\]
Since $\sum_1^s q_i g^i\in I$, $f=\sum_{s+1}^t q_i g^i$ in $J/I$.
Proposition \ref{nr4}  implies uniqueness. \end{proof} 

\begin{cor}\label{nr7}
If $I\subset J$ are ideals in $\Opis_{K,0}$, then
\[\dim_K(J/I)=\sum_{i=s+1}^t \# \Delta_i=\# (\Npis(J)\setminus\Npis(I)).\]
\end{cor}

Applying the Nakayama Lemma one may prove

\begin{prop}\label{nr8}
Assume that $I\subset J$ are ideals in $\Opis_{K,0}$.

Then $\dim_K (J/I)<\infty$ if and only if there exists $k$ such that $\bm_0^k J\subset I$. 
\end{prop}

From Corollaries \ref{nr2}, \ref{nr7} and Proposition \ref{nr8} we get 

\begin{cor}\label{nr9}
Assume that $I\subset J$ are ideals in $\Opis_{K,0}$. The following conditions are equivalent:
\begin{itemize}
\item[(i)] $\dim_K (J/I)<\infty$,
\item[(ii)] $\# (\Npis(J)\setminus\Npis(I))<\infty$,
\item[(iii)] there exists $k$ such that  $\bm_0^k J\subset I$,
\item[(iv)] there is an open neighbourhood $W$ of the origin in $\C^n$
such that $I_p=J_p$ for $p\in W\setminus\{\bO\}$.
\end{itemize}

If that is the case and both ideals are proper then $V_C(I)=V_C(J)$. 
\end{cor}

Applying the local Nullstellensatz we get

\begin{cor}\label{nr10}
Assume that $I\subset\Opis_{K,0}$ is an ideal. The following conditions are equivalent
\begin{itemize}
\item[(i)] $\dim_K (\Opis_{K,0}/I)<\infty$,
\item[(ii)] $\#(\N^n\setminus\Npis(I))<\infty$,
\item[(iii)] $V_C(I)\subset\{\bO\}$ in some neighbourhood of the origin. 
\end{itemize}
\end{cor}

Let $f_1,\ldots,f_m\in\Opis_{R,0}$, and let
\[h_i^a=\sum_{j=1}^m A_{ij}(a)f_j,\ 1\leq i\leq \ell,\]
where $A_{ij}\in\R[a_1,\ldots,a_d]$ and $a=(a_1,\ldots,a_d)$.

For $a\in\K^d$, let $I_K^a$ denote the ideal in $\Opis_{K,0}$
generated by $h_1^a,\ldots,h_\ell^a$.
Applying arguments presented in (\cite{bierstonemilman}, Chapter I) one may prove
\begin{prop}\label{nr101}

\begin{itemize}
\item[(i)] if $a\in\R^d$ then $\Npis(I_R^a)=\Npis(I_C^a)$,
\item[(ii)] there exists a proper algebraic set $\Sigma_K\subset\K^d$, defined by
polynomials with real coefficients, such that
$\Npis(I_K^a)$ is constant for $a\in\K^d\setminus\Sigma_K$. 
In particular, $\Sigma_R=\Sigma_C\cap \R^d$ is a proper
real algebraic set and $\Npis(I_R^a)$ is constant for $a\in\R^d\setminus\Sigma_R$. 
\end{itemize}
\end{prop}

Let $J_K$ denote the ideal in $\Opis_{K,0}$  generated 
by $f_1,\ldots, f_m$. Each $I_K^a\subset J_K$, and then $\Npis(I_K^a)\subset\Npis(J_K)$.

The function $\K^d\ni a\mapsto\Npis(I_K^a)$ is upper-semicontinuous
(see \cite{bierstonemilman} for details), so we have

\begin{prop}\label{nr11}
If there exists $a\in\K^d$ such that
$\#(\Npis(J_K)\setminus\Npis(I_K^a))<\infty$,
then 
$\#(\Npis(J_K)\setminus\Npis(I_K^a))<\infty$ for all $a\in\K^d\setminus\Sigma_K$. 
\end{prop}

\section{Curves having isolated singularity}
Let $f_1,\ldots, f_m\in\bm_0\cap\Opis_{K,0}$. The next proposition is a consequence
of well-known properties of complex analytic germs
(see \cite{narasimhan}, Proposition 4, p.46, and Proposition 13, p.60).

\begin{prop}\label{nr12}
$\dim V_C(f_1,\ldots,f_m)\leq 1$ if and only if there exists
a germ $a\in \bm_0$ such that $V_C(f_1,\ldots,f_m,a)=\{\bO\}$.

If that is the case and $\bO$ is not isolated in $V_C(f_1,\ldots,f_m)$,
then $V_C(f_1,\ldots,f_m)$ is a 1-dimensional germ, i.e. it is a complex
curve in some neighbourhood of the origin. 
\end{prop}

Suppose that $\dim V_C(f_1,\ldots,f_m)\leq 1$. By
\[D(f_1,\ldots,f_m)(p)=\left[ \frac{\partial f_j}{\partial z_i}(p)\right]\]
we shall denote the Jacobian matrix at $p$.
We shall say that $V_C(f_1,\ldots, f_m)$ is a curve having an isolated singularity
at the origin if 
\[\{ p\in V_C(f_1,\ldots,f_m)\ |\ \operatorname{rank}\, D(f_1,\ldots,f_m)(p)<n-1\}\subset\{\bO\}\]
in some neighbourhood of the origin. 

If that is the case and $\dim V_C(f_1,\ldots, f_m)=1$, then
\[\operatorname{rank}\, D(f_1,\ldots,f_m)\equiv n-1\mbox{ on }V_C(f_1,\ldots,f_m)\setminus\{\bO\}\ .\]
We want to point out that if $V_C(f_1,\ldots,f_m)=\{\bO\}$ then we also call it a curve having
an isolated singularity, despite that $\dim V_C(f_1,\ldots,f_m)= 0$.

Let $I\subset \Opis_{K,0}$ denote the ideal generated by $f_1,\ldots, f_m$
and all $(n-1)\times(n-1)$--minors of $D(f_1,\ldots,f_m)$. By  Corollary \ref{nr10} we get

\begin{prop}\label{nr13}
$V_C(f_1,\ldots, f_m)$ is a curve having an isolated singularity at the origin
if and only if $\dim_K \Opis_{K,0}/I<\infty$,
i.e. if $\#(\N^n\setminus\Npis(I))<\infty$. 
\end{prop}

\begin{lemma}\label{nr16}
Let $g_1,\ldots,g_{n-1},h\in\left<f_1,\ldots,f_m\right>\subset\Opis_{K,0}$.
Assume that curves $V_C(f_1,\ldots,f_m)$ and $V_C(g_1,\ldots,g_{n-1})$ have an isolated singularity at the origin.

Then $V_C(f_1,\ldots,f_m)=V_C(g_1,\ldots,g_{n-1},h)$ if and only if
\[\dim_K \left<f_1,\ldots,f_m\right>/\left<g_1,\ldots,g_{n-1},h\right>\ <\infty.\]
\end{lemma}
\begin{proof} ($\Leftarrow$) is a consequence of Corollary \ref{nr9}.\\
($\Rightarrow$) Take $p\in V_C(f_1,\ldots,f_m)\setminus\{\bO\}=V_C(g_1,\ldots,g_{n-1},h)\setminus\{\bO\}$
near the origin. The germ of $V_C(g_1,\ldots,g_{n-1},h)$ at $p$ is one--dimensional, so
\[\operatorname{rank}\, D(g_1,\ldots,g_{n-1})(p)\leq\operatorname{rank}\, D(g_1,\ldots,g_{n-1},h)(p)\leq n-1.\]
As $V_C(f_1,\ldots,f_m)$ and $V_C(g_1,\ldots,g_{n-1})$ have an isolated singularity,
\[n-1=\operatorname{rank}\, D(f_1,\ldots,f_m)(p)=\operatorname{rank}\, D(g_1,\ldots,g_{n-1})(p),\]
and then $\operatorname{rank}\, D(f_1,\ldots,f_m)(p)=\operatorname{rank}\, D(g_1,\ldots,g_{n-1},h)(p)=n-1$.

Then ideals generated by representatives of $f_1,\ldots,f_m$ and
$g_1,\ldots,g_{n-1},h$ in $\Opis_{C,p}$ are equal. By Corollary \ref{nr9} (i)(iv),
\[\dim_K\left<f_1,\ldots,f_m\right>/\left<g_1,\ldots,g_{n-1},h\right>\ <\infty .\] \end{proof}

\begin{lemma}
If $g_1,\ldots,g_\ell \in\left< f_1,\ldots,f_m \right>$ and $V_C(g_1,\dots,g_\ell)$
has an isolated singularity at the origin, then $V_C(f_1,\ldots,f_m)$ has an isolated singularity too.
\end{lemma}
\begin{proof}  Of course $V_C(f_1,\ldots,f_m)\subset V_C(g_1,\ldots,g_\ell)$.
If $V_C(g_1,\ldots,g_\ell)\subset\{\bO\}$ then the conclusion is obvious.

If this is not the case, then $\dim V_C(f_1,\ldots,f_m)\leq\dim V_C(g_1,\ldots,g_\ell)=1$.
For $p\in V_C(f_1,\ldots,f_m)$, each gradient $\nabla g_i(p)$ is a linear combination
of gradients $\nabla f_j(p)$, and then
\[\operatorname{rank}\, D(g_1,\ldots,g_\ell)(p)\leq\operatorname{rank}\, D(f_1,\ldots,f_m)(p).\]
As $V_C(g_1,\dots,g_\ell)$ has an isolated singularity, then
$\operatorname{rank}\, D(g_1,\ldots,g_\ell)(p)=n-1$ for all
$p\in V_C(g_1,\ldots,g_\ell)\setminus\{\bO\}$ lying near the origin, which implies
that $\operatorname{rank}\, D(f_1,\ldots,f_m)(p)\geq n-1$
for $p\in V_C(f_1,\ldots,f_m)\setminus\{\bO\}$. \end{proof} 

\begin{prop}\label{wstawka1}
If $g_1,\ldots,g_{n-1}\in\left< f_1,\ldots,f_m \right>$, $h\in\Opis_{K,0}$ and
\[V_C\left(\frac{\partial(h,g_1,\ldots,g_{n-1})}{\partial(x_1,x_2,\ldots,x_n)},g_1,\ldots,g_{n-1}\right)
\subset\{\bO\}\]
then $V_C(f_1,\ldots,f_m)$ has an isolated singularity at the origin.
\end{prop}
\begin{proof}  If $p\in V_C(g_1,\ldots,g_{n-1})\setminus\{\bO\}$ lies near the origin
then the determinant of
$ D(h,g_1,\ldots,g_{n-1})(p)$ does not vanish,
and then $\operatorname{rank}\, D(g_1,\ldots,g_{n-1})(p)=n-1$.
So $V_C(g_1,\ldots,g_{n-1})$ has an isolated singularity, and one may apply the previous Lemma. \end{proof} 

Let $\Mpis(k,m;\K)$ denote the space of all $k\times m$--matrices with coefficients in $\K$.

\begin{lemma}\label{nr14}
Assume that $S$ is a finite set,  and $v_{s1},\ldots,v_{sm}\in\C^n$,
where $s\in S$, is a finite collection of vectors. Let $V_s$ denote the $\C$-linear space spanned
by $v_{s1},\ldots,v_{sm}$. Assume that each $\dim_C V_s\geq k$.

Then there exists a proper algebraic subset $\Sigma_1\subset\Mpis(k,m;\K)$, defined by polynomials
with real coefficients, such that for every $[a_{tj}]\in\Mpis(k,m;\K)\setminus\Sigma_1$,
each $\C$--linear space  spanned by
\[w_{st}=\sum_{j=1}^m a_{tj} v_{sj},\ \ 1\leq t\leq k,\]
is $k$--dimensional for each $s\in S$. 
\end{lemma}

\begin{prop}\label{15}
Assume that  $V_C(f_1,\ldots, f_m)$ 
is a curve having an isolated singularity at the origin.

Then there exists a proper algebraic subset
$\Sigma_2\subset\Mpis_K=\Mpis(n-1,m;\K)$, defined by polynomials with real coefficients,
such that for every $(n-1)\times m$--matrix $a =[a_{tj}]\in\Mpis_K\setminus\Sigma_2$ and
\[g_t^a=\sum_{j=1}^m a_{tj} f_j,\ \ 1\leq t\leq n-1,\]
one has
$\{p\in V_C(f_1,\ldots,f_m)\ |\ \operatorname{rank}\, D(g_1^a,\ldots,g_{n-1}^a)(p)<n-1\}\subset\{\bO\}$.

\end{prop}
\begin{proof}  For $a=[a_{tj}]\in\Mpis_K$, 
let $J^a\subset\Opis_{K,0}$ denote the ideal generated by $f_1,\ldots,f_m$ and all
$(n-1)\times (n-1)$--minors of $D(g_1^a,\ldots,g_{n-1}^a)$.
Then 
\[\{p\in V_C(f_1,\ldots,f_m)\ |\ \operatorname{rank}\, D(g_1^a,\ldots,g_{n-1}^a)(p)<n-1\}\subset\{\bO\}\]
if and only if $\#(\N^n\setminus\Npis(J^a))<\infty$.

From Proposition \ref{nr101}, there exists a proper algebraic subset $\Sigma_2\subset {\cal M}_K$,
defined by polynomials with real coefficients, such that $\Npis(J^a)$ is constant for
$a\in {\cal M}_K\setminus\Sigma_2$.
By Proposition \ref{nr11}, it is enough to find at least one $a\in\Mpis_K$ such that
$\# (\N^n\setminus\Npis(J^a))<\infty$.

The curve $V_C(f_1,\ldots,f_m)$ is a finite union of complex irreducible curves
$\{ C_s \}$, and each $C_s\setminus\{\bO\}$ is locally biholomorphic to $\C\setminus\{\bO\}$.
Take $p_s\in\C_s\setminus\{\bO\}$ near the origin. Put
\[v_{sj}=\nabla f_j(p_s),\ 1\leq j\leq m.\]
The dimension of the linear space spanned by $v_{s1},\ldots, v_{sm}$ equals\\
$\operatorname{rank}\, D(f_1,\ldots,f_m)(p_s)= n-1$.
By Lemma \ref{nr14} there exists a proper algebraic subset
$\Sigma_1\subset\Mpis_K$, defined by polynomials with real coefficients,
such that for every $a=[a_{tj}]\in\Mpis_K\setminus\Sigma_1$
the dimension of the space spanned by all
\[\sum_{j=1}^m  a_{tj}\nabla f_j(p_s)=\nabla g_t^a(p_s)\]
equals $n-1$ for each $s$. If that is the case then
 $\operatorname{rank}\, D(g_1^a,\ldots,g_{n-1}^a)(p_s)=n-1$,
and then the same equality holds for all $s$ and all $p\in\C_s\setminus\{\bO\}$ lying near $\bO$,
i.e. for all $p\in V_C(f_1,\ldots,f_m)\setminus\{\bO\}$
sufficiently close to the origin. \end{proof}

\begin{theorem}\label{nr17}
Assume that $V_C(f_1,\ldots,f_m)$ is a curve with an isolated 
singularity.

Then there exists a proper algebraic subset
$\Sigma_3\subset\Mpis(n-1,m;\K)\times\K^m=\Mpis_K\times\K^m$,
defined by polynomials with real coefficients, such that for every
$(a,b)=\left( [a_{tj}],(b_1,\ldots,b_m)\right)\not\in\Sigma_3$ and
\[g_t^a=\sum_{j=1}^m a_{tj} f_j,\ \ 1\leq t\leq n-1,\]
\[h^b=b_1 f_1+\cdots+b_m f_m\ ,\]
\begin{itemize}
\item[(i)] $V_C(g_1^a,\ldots,g_{n-1}^a)$ is a curve with an isolated singularity at the origin,
\item[(ii)] $V_C(g_1^a,\ldots,g_{n-1}^a,h^b)=V_C(f_1,\ldots,f_m)$.

\end{itemize}
\end{theorem}
\begin{proof}   Let $J_K$ denote the ideal in $\Opis_{K,0}$ generated by
$f_1,\ldots,f_m$. For $(a,b)\in\Mpis_K\times\K^m$ let $I^a$ denote the ideal generated
by $g_1^a,\ldots,g_{n-1}^a$ and all $(n-1)\times(n-1)$--minors of $D(g_1^a,\ldots,g_{n-1}^a)$,
and let $I^{a,b}$ denote the one generated by 
germs $g_1^a,\ldots,g_{n-1}^a,h^b$.
 
By Proposition \ref{nr101} there exists a proper algebraic set $\Sigma_3\subset\Mpis_K\times\K^m$,
defined by polynomials with real coefficients, such that
$\Npis(I^{a})$, as well as $\Npis(I^{a,b})$, is constant for all $(a,b)\not\in\Sigma_3$.

By Proposition \ref{nr11}, \ref{nr13} and Lemma \ref{nr16}, it is enough to find at least one
$(a,b)\in\Mpis_C\times\C^m$ such that
and $\# (\N^n\setminus \Npis(I^a))<\infty $ and $\# (\Npis(J_C)\setminus\Npis(I^{a,b}))<\infty$.\\
{\em (i)} Let $U=\{p\in\C^n\ |\ f_1(p)\neq 0\}$. Let ${\cal D}$ denote the space of all complex
$(n-1)\times (m-1)$--matrices $d=[d_{tj}]$, where $1\leq t\leq n-1$, $2\leq j\leq m$,
and let
\[a_{t1}(p,d)=-\left(\sum_{j=2}^m d_{tj} f_j(p)\right)\cdot (f_1(p))^{-1},\ \ \ 
a_{tj}(p,d)=d_{tj}\ .\]
Define a mapping
\[U\times {\cal D}\ni (p,d)\, \mapsto\, A(p,d)=[a_{tj}(p,d)]\in\Mpis_C\ .\]
Let $a=[a_{tj}]\in\Mpis_C$ be a regular value of $A$. Then $(p,d)\in A^{-1}(a)$ if and only if
$p\in U$, $a_{tj}=d_{tj}$ for $1\leq t\leq n-1$, $2\leq j\leq m$, and
\[g_t^a(p)=\sum_{j=1}^m a_{tj} f_j(p)=0\ .\]

If that is the case then the rank of the Jacobian matrix of $A$ at $(p,d)$
equals $(n-1)m=n-1+\dim {\cal D}$, and then
\[\rank\, \left[\frac{\partial a_{t1}}{\partial z_i}(p,d)\right]=n-1,\]
where $1\leq t\leq n-1, 1\leq i\leq n$.
At $(p,d)$ we have
\[-\frac{\partial a_{t1}}{\partial z_i}\cdot f_1^2=
\left( \sum_{j=2}^m d_{tj}\frac{\partial f_j}{\partial z_i}\right)f_1-\left(\sum_{j=2}^m d_{tj} f_j\right)\frac{\partial f_1}{\partial z_i}=\]
\[\left(\sum_{j=2}^m a_{tj}\frac{\partial f_j}{\partial z_i}\right)f_1+a_{t1}
\cdot f_1\cdot\frac{\partial f_1}{\partial z_i}=
\left(\sum_{j=1}^m a_{tj}\frac{\partial f_j}{\partial z_i}\right)\cdot f_1=
\frac{\partial g_t^a}{\partial z_i}\cdot f_1 .\]
As $f_1(p)\neq 0$, if $g_1^a(p)=\ldots=g_{n-1}^a(p)=0$ then

\[\operatorname{rank}\, D(g_1^a,\ldots,g_{n-1}^a)(p)=
\operatorname{rank}\, \left[ \frac{\partial g_t^a}{\partial z_i}(p) \right]= n-1\ .\]

By the Sard theorem, one may choose $a\in\Mpis_C$ such that the above condition holds for all points
\[p\in V_C(g_1^a,\ldots,g_{n-1}^a)\setminus V_C(f_1,\ldots,f_m)=\]
\[\bigcup_{k=1}^m  \left( V_C(g_1^a,\ldots,g_{n-1}^a)\setminus V_C(f_k)\right).\]

Let $\Sigma_2\subset\Mpis_C$ be as in Proposition \ref{15}. One may choose $a\not\in\Sigma_2$, so that the origin is isolated in
\[\{p\in V_C(f_1,\ldots,f_m)\ |\ \operatorname{rank}\, D(g_1^a,\ldots,g_{n-1}^a)(p)<n-1\}.\]
Hence the curve $V_C(g_1^a,\ldots,g_{n-1}^a)$ has an isolated singularity at the origin,
i.e. $\#(\N^n\setminus\Npis(I^a))<\infty$.\\[1em]
{\em (ii)} Take any nontrivial $h^b=b_1 f_1+\cdots+ b_m f_m$, $b_i\in\K$.
Then $\dim V_C(h^b)=n-1$.
There exists a finite collection  $\{V_s\}$ of analytic manifolds such that
$V_C(h^b)\setminus V_C(f_1)=\bigcup_s V_s$. In particular, each
$\dim V_s\times {\cal D}\leq n-1+(n-1)(m-1)=\dim \Mpis_C$.

We may assume that a matrix $a\in\Mpis_C$ is a regular value for each restricted mapping
$A|V_s\times{\cal D}$. 
As there is a one-to-one correspondence between points in $A^{-1}(a)\cap V_s\times{\cal D}$
and  $V_C(g_1^a,\ldots, g_{n-1}^a)\cap V_s$, both are either void or discrete.
In consequence,
$V_C(g_1^a,\ldots,g_{n-1}^a,h^b)\setminus V_C(f_1)=\bigcup_s V_C(g_1^a,\ldots,g_{n-1}^a)\cap V_s$
is discrete too. Applying the same arguments for each $V_C(h^b)\setminus V_C(f_j)$, we can find
a matrix $a\in\Mpis_C$ and associated germs $g_1^a,\ldots,g_{n-1}^a$ such that
$V_C(g_1^a,\ldots,g_{n-1}^a,h^b)\setminus V_C(f_1,\ldots,f_m)$ is discrete.

Then the corresponding germs at the origin satisfy an inclusion
\[V_C(g_1^a,\ldots,g_{n-1}^a,h^b)\subset V_C(f_1,\ldots,f_m).\]
As $g_1^a,\ldots,g_{n-1}^a,h^b\in\left< f_1,\ldots, f_m\right>$, then
$V_C(g_1^a,\ldots,g_{n-1}^a,h^b)=V_C(f_1,\ldots,f_m)$.
By Lemma \ref{nr16},
\[\dim_K \left<f_1,\ldots,f_m\right>/\left<g_1^a,\ldots,g_{n-1}^a,h^b\right>=\#(\Npis(J_C)\setminus\Npis(I^{a,b}))<\infty.\] \end{proof}

\section{The number of half-branches}
Assume that $f_1,\ldots,f_m\in\bm_0\cap\Opis_{R,0}$,
$\operatorname{rank}\, D(f_1,\ldots,f_m)(\bO)<n-1$,
and $V_C(f_1,\ldots,f_m)$ is a curve with an isolated singularity at the origin.

By Theorem \ref{nr17} there exists $g_1,\ldots,g_{n-1},h\in\left< f_1,\ldots,f_m\right>\cap\Opis_{R,0}$
such that $V_C(g_1,\ldots,g_{n-1})$ is a curve with an isolated singularity,
$\operatorname{rank}\, D(g_1,\ldots,g_{n-1})(\bO)<n-1$,
and $V_C(g_1,\ldots,g_{n-1},h)=V_C(f_1,\ldots,f_m)$.
In particular, the number of half-branches of $V_R(f_1,\ldots,f_m)$ emanating from the origin
is the same as the number $b_0$ of half-branches of $V_R(g_1,\ldots,g_{n-1})$
on which $h$ vanishes.

Now we recall the formula for $b_0$ presented in \cite{szafraniec14}.
Let $J$ denote the ideal generated by $f_1,\ldots,f_m$.
Let $J_k$, where $k=1,2$, denote the ideal generated by $g_1,\ldots,g_{n-1},h^k$.
Of course, $V_C(J)=V_C(J_1)=V_C(J_2)$. By Lemma \ref{nr16},
$\dim_R (J/J_1)<\infty$ and $\dim_R(J/J_2)<\infty$.
Since $J_2\subset J_1\subset J$, $\dim_R (J_1/J_2)<\infty$,
and then $\# (\Npis(J_1)\setminus\Npis(J_2))<\infty$.

Put
$\xi = 1+\max\, \{|\beta|\}-\min\, \{|\alpha|\}$,
where $\beta\in\Npis(J_1)\setminus\Npis(J_2)$ and $\alpha\in\Bpis(J_1)\setminus\Npis(J_2)$.
(We take $\xi=1$ if $\Bpis(J_1)=\Bpis(J_2)$.)
Let $k>\xi$ be an even positive integer,
and let $\omega:\R^n,\bO\rightarrow\R^n,\bO$ be a non-negative polynomial
which is $k$--flat at the origin
and $V_R(g_1,\ldots,g_{n-1},\omega)=\{\bO\}$.
Let
\[h_+=\frac{\partial(h+\omega,g_1,\ldots,g_{n-1})}{\partial(x_1,x_2,\ldots,x_n)},\]
\[h_-=\frac{\partial(h-\omega,g_1,\ldots,g_{n-1})}{\partial(x_1,x_2,\ldots,x_n)},\]
\[H_+=(h_+,g_1,\ldots,g_{n-1}):\R^n,\bO\rightarrow \R^n,\bO\ ,\]
\[H_-=(h_-,g_1,\ldots,g_{n-1}):\R^n,\bO\rightarrow \R^n,\bO\ .\]

As an immediate consequence of Theorem 2.5 \cite{szafraniec14} we get

\begin{theorem}\label{nr18}
Assume that $f_1,\ldots,f_m\in\bm_0\cap\Opis_{R,0}$,
$\operatorname{rank}\, D(f_1,\ldots,f_m)(\bO)<n-1$,
and $V_C(f_1,\ldots,f_m)$ is a curve having an isolated singularity at the origin.

Then the origin is isolated in $H_+^{-1}(\bO)$ and $H_-^{-1}(\bO)$, so that the local
topological degree of $\deg_0(H_+)$ and $\deg_0(H_-)$ is defined.
Moreover,
\[b_0=\deg_0(H_+)-\deg_0(H_-)\ ,\]
where $b_0$ is the number of real half-branches in $V_R(f_1,\ldots,f_m)$
emanating from the origin.
\end{theorem}

\begin{prop}\label{nr19}
If $\dim \left< f_1,\ldots,f_m\right> /\left< g_1,\ldots,g_{n-1}\right>  <\infty$
and both $V_C(f_1,\ldots,f_m)$, $V_C(g_1,\ldots,g_{n-1})$ have an isolated singularity
at the origin then $b_0=2\deg_0(H_1)$, where
\[H_1=\left(\frac{\partial(\Omega,g_1,\ldots,g_{n-1})}{\partial(x_1,x_2,\ldots,x_n)},g_1,\ldots,g_{n-1}\right):\R^n\bO\rightarrow\R^n,\bO\]
is a mapping having an isolated zero at $\bO$ and $\Omega=x_1^2+\ldots+x_n^2$. 
\end{prop}
\begin{proof}  By Lemma \ref{nr16} , $b_0$ equals the number of half-branches
in  $V_R(g_1,\ldots,g_{n-1})$ emanating from the origin. So according to \cite{aokietal1,aokietal2} ,
$b_0=2\deg_0(H_1)$. \end{proof} 

\begin{ex} Let $f_1(x,y)=x^3$, $f_2(x,y)=x(x-y)$, and let $I_1=\left< f_1,f_2,y\right>$.
One may check that $\Bpis(I_1)=\{(2,0),(0,1)\}$, so that $\#(\N^2\setminus\Npis(I_1))=2$.
By Corollary \ref{nr10} and Proposition \ref{nr12}, $\dim V_C(f_1,f_2)\leq 1$.

Let $I_2$ be the ideal generated by $f_1,f_2$ and all $1\times 1$-minors of $D(f_1,f_2)$, i.e.
$I_2=\left< f_1,f_2,3x^2,2x-y,-x  \right>$. One may check that
$\Bpis(I_2)=\{(1,0),(0,1)\}$,
so that $\#(\N^2\setminus\Npis(I_2))=1$. By Proposition \ref{nr13},
$V_C(f_1,f_2)$ has an isolated singularity at the origin.

Let $J=\left< f_1,f_2\right> $. One may check that $\Bpis(J)=\{(2,0),(1,2)\}$.

Take $g_1=f_1-6 f_2$. Let $I_3$ be the ideal generated by $g_1$ and all
$1\times 1$-minors of $D(g_1)$, i.e. $I_3=\left< g_1,3x^2-12 x+6 y, 6 x\right>$.
One may check that $\Bpis(I_3)=\{(1,0),(0,1)\}$, so that
$V_C(g_1)$ has an isolated singularity at the origin.

Take $h=f_1+5 f_2$. Let $J_1=\left<g_1,h\right>$, $J_2=\left< g_1,h^2\right>$.
One may check that
$\Bpis(J_1)=\{(2,0),(1,2)\}$, $\Bpis(J_2)=\{(2,0),(1,5)\}$.
We have $J_1\subset J$ and $\Npis(J)\setminus\Npis(J_1)=\emptyset$, so
that $\dim_R (J/J_1)=0$. By Lemma \ref{nr16},
$V_C(f_1,f_2)=V_C(g_1,h)$. Moreover, $\Bpis(J_1)\setminus\Npis(J_2)=\{(1,2)\}$.

As $\Npis(J_1)\setminus\Npis(J_2)=\{(1,2),(1,3),(1,4)\}$, then
\[\max \{|\beta|\}=5,\mbox{ where }\beta\in\Npis(J_1)\setminus\Npis(J_2),\]
\[\min\{|\alpha|\}=3, \mbox{ where }\alpha\in\Bpis(J_1)\setminus\Npis(J_2),\]
\[\xi=1+5-3=3\ .\]
Take $k=4>\xi=3$.  Of course $\omega=x^4+y^4$ is a non-negative polynomial which is $4$-flat at the origin
and $V_R(g_1,\omega)=\{\bO\}$. Set
\[h_+=\frac{\partial(h+\omega,g_1)}{\partial(x,y)},\ \ h_-=\frac{\partial(h-\omega,g_1)}{\partial(x,y)},\]
\[H_+=(h_+,g_1),\ \ H_-=(h_-,g_1).\]
Using the computer program written by Andrzej Łęcki one may compute
\[\deg_0(H_+)=1,\ \ \deg_0(H_-)=-1\ .\]
According to Theorem \ref{nr18}, there are two half-branches in $V_R(f_1,f_2)$ emanating from the origin.

It is worth to notice that for any $g=a f_1+b f_2$, where $a,b\in\R$,
if $b\neq 0$ then there are four half-branches in $V_R(g)$ emanating from the origin.
If $b=0$ then $V_C(g)$ does not have an isolated singularity at the origin. 
So in both cases one cannot apply the apparently simpler Proposition \ref{nr19}. 
\end{ex}

\begin{ex} Let $v=(v_1,v_2,v_3)=(x^2-y^2z,y^2-xz,z^3+x^3)$, $w=(w_1,w_2,w_3)=(z^2+xy,x^2+yz^2,y^3-x^2z)$ be vector fields in $\R^3$. Vectors $v(p)$ and $w(p)$ are co--linear if and only if $f_1=v_2w_3-v_3w_2$, $f_2=v_1w_3-v_3w_1$, $f_3=v_1w_2-v_2w_1$ vanish at $p$. Put $g_1=f_1+f_2+f_3$, $g_2=f_1-f_2$, $h=f_3$. 

Let $J_1=\left< g_1,g_2,h \right>$, $J_2=\left< g_1,g_2,h^2\right>$. One may check that\\
$\Bpis(J_1)=\{ (4,0,0),(2,3,0),(1,4,0),(0,5,1) \}$,\\ 
$\Bpis(J_2)=\{ (4,0,0),(2,3,0),(1,6,0),(1,5,2),(0,7,2),(1,4,5),(0,6,5),(0,5,6),$\\ $(0,12,1) \}$,\\
so that
\[\max\{|\beta|\}=12,\mbox{ where }\beta\in\Npis(J_1)\setminus\Npis(J_2),\]
\[\min\{|\alpha|\}=5,\mbox{ where }\alpha\in\Bpis(J_1)\setminus\Npis(J_2),\]
\[\xi=1+12-5=8.\]
Take $k=10>\xi=8$. Of course, $\omega=x^{10}+y^{10}+z^{10}$ is a non-negative polynomial
which is 10-flat at the origin and $V_R(g_1,g_2,\omega)=\{\bO\}$. Set $h_+,h_-,H_+,H_-$ as in
Theorem \ref{nr18}. One may check that $\dim\Opis_{C,0}/\left< h_{\pm},g_1,g_2\right><\infty$, so that
$V_C(h_{\pm},g_1,g_2)=\{\bO\}$. By Proposition \ref{wstawka1}, $V_C(g_1,g_2)$, as well as $V_C(f_1,f_2,f_3)$,
has an isolated singularity at the origin.

Using the computer program one may compute $\deg_0(H_+)=3$, $\deg_0(H_-)=-1$.
According to Theorem \ref{nr18}, there are 4 half-branches emanating from the origin
in the set where vector fields $v$ and $w$ are co-linear.
\end{ex}

\section{Germs from $\R^2$ to $\R^3$} \label{sec_ip} 

Mappings from $(\K^2,\bO)$ to $(\K^3,\bO)$ are a natural object of study in the theory of singularities.
In \cite{mond1} Mond has classified simple smooth germs from $(\R^2,\bO)$ to $(\R^3,\bO)$,
in \cite{mond2} he was investigating multiple points of complex germs from
$(\C^2,\bO)$ to $(\C^3,\bO)$. Some constructions presented in this section are similar
to the ones in \cite{mond2}.

In \cite{mararnunoballesteros}   Marar and Nu\~{n}o-Ballesteros study finitely determined map germs $u$ from
$(\R^2,\bO)$ to $(\R^3,\bO)$ and the curve obtained as the intersection of its image
with a sufficiently small sphere centered at the origin (the associated doodle
of $u$). The set of singular points of the image of $u$ is the closure of the double points curve,
denoted by $D^2(u)$. 
They proved \cite[Theorem 4.2]{mararnunoballesteros}, that if $u$
is finitely determined, has type $\Sigma ^{1,0}$, 
and its double point curve $D^2(u)$ has $r$ real half--branches, its associated doodle
is equivalent to $\mu _r\colon [0;2\pi]\ar \R^2$ given by $\mu _r(t)=(\sin
t,\sin rt)$ if $r$ is even, or $\mu _r(t)=(\sin t,\cos rt)$ if $r$ is odd.
Moreover $u$ is topologically equivalent to the map germ
$M_r(x,y)=(x,y^2,\operatorname{Im}((x+iy)^{r+1}))$, where
$\operatorname{Im}(z)$ denotes the imaginary part of $z\in \C$ \cite[Corollary
4.5]{mararnunoballesteros}. 

If $u$ is finitely determined, has fold type, its 2--jet belongs to the orbit $(x,xy,0)$, and its double point curve $D^2(u)$ has $r$ real half--branches, then $u$ is topologically equivalent to the map germ $M_r$ \cite[Theorem 5.3]{mararnunoballesteros}. 

Moreover, if $u$ has no triple points then \cite[Lemma 4.1]{mararnunoballesteros} it has type $\Sigma ^{1,0}$. 

\bigskip
In the remainder of this section we study analytic germs which are not necessarily finitely determined.

\medskip

Assume that
$u=(u_1,u_2 ,u_3)\colon(\K ^{2},\bO) \ar (\K ^{3},\bO)$ is an analytic germ.
For $x=(x_1,x_2)$, $y=(y_1,y_2)$, $1\leqslant i \leqslant 3$, define 
\[
w_i(x,y)=u_i(x)-u_i(y),
\]
\[w=(w_1,w_2,w_3)\colon\K^2\times\K^2=\K ^4 \ar \K^3.\]
Then $w$ is also analytic, and there
exist analytic germs $c_{ik}$ such that
\[w_i(x,y)=c_{i1}(x,y)(x_1-y_1)+
c_{i2}(x,y)(x_2-y_2).\] The germs $c_{ik}$ are not uniquely determined. 

We  fix germs $c_{ik}$, and for $1\leqslant i<j \leqslant 3$ we define
\[ W_{ij}=
\left | \begin{array}{cc}
c_{i1} & c_{i2} \\
c_{j1} & c_{j2}
\end{array} \right |. \]

%**************** LEMMA cij(x,x) **********************
\begin{lemma} \label{cij(x,x)}
$c_{ik}(x,x)=\frac{\partial u_i}{\partial x_k}(x)$, and then 
$W_{ij}(x,x)=\frac{\partial(u_i,u_j)}{\partial(x_1,x_2)}(x)$.
In particular, if $u$ has a critical point at the origin then all $W_{ij}(\bO)=0$.
\end{lemma}
\begin{proof} 
Put
\[\frac{\partial}{\partial z_k}:=\frac{1}{2}\left ( \frac{\partial}{\partial x_k}-
\frac{\partial}{\partial y_k}\right ).\]
Then
\[\frac{\partial}{\partial z_k}(x_r-y_r)=\begin{cases} 0, & k\neq r\\ 1, & k=r \end{cases}\ ,\]

\[\frac{\partial w_i}{\partial z_k}=\left (\frac{\partial c_{i1}}
{\partial z_k}(x_1-y_1)+\frac{\partial c_{i2}}
{\partial z_k}(x_2-y_2)\right )+c_{ik}\ ,\]
\[\frac{\partial w_i}{\partial z_k}(x,x)=c_{ik}(x,x).\]
On the other hand we have
\begin{align*}
\frac{\partial w_i}{\partial z_k}(x,y) & =\frac{1}{2}\left ( \frac{\partial u_i}{\partial x_k}(x)+\frac{\partial u_i}{\partial x_k}(y)\right ), \\
\frac{\partial w_i}{\partial z_k}(x,x) & =\frac{\partial u_i}{\partial x_k}(x).
\end{align*}
Thus
\[ c_{ik}(x,x)=\frac{\partial u_i}{\partial x_k}(x).\] \end{proof}

\bigskip

By Cramer's rule
\begin{equation}\label{wzKramer} (x_1-y_1)W_{ij}=
\left | \begin{array}{cc}
w_{i} & c_{i2} \\
w_{j} & c_{j2} 
\end{array} \right |, 
\qquad
(x_2-y_2)W_{ij}=
\left | \begin{array}{cc}
c_{i1} & w_{i}\\
c_{j1} & w_{j}  
\end{array} \right |. \end{equation}

\medskip

Let us assume that $u$ has an isolated critical point at the origin, i.e.
there exists a neighbourhood of the origin where the derivative matrix of $u$
has rank $2$, except at the origin. Then $u|\K ^2\setminus \{\bO\}$ is a germ of
an immersion. In particular, $u$ has an isolated singularity if the origin is an isolated
zero of the ideal generated by all $2\times 2$--minors of the derivative matrix of $u$

Notice that the set $W=w\inv (\bO)$ contains all the self--intersection points $(p,q)$ of $u$. More precisely, if
$\Delta =\{ (x,x)\ | \ x\in \K^2\}$, 
then $W\setminus \Delta$ consists of all the points $(p,q)$ and $(q,p)$, where $u(p)=u(q)$
 is the double point of $u$. 

Put $f=(f_1,\ldots,f_6)=(w_1,w_2,w_3,W_{12},W_{13},W_{23}):\K^4=\K^2\times\K^2\rightarrow \K^6$ and
\[ V  =W\cap \bigcap _{1\leqslant i<j \leqslant 3} W_{ij}\inv (\bO)=f^{-1}(\bO)=V(f_1,\ldots,f_6)\]
where by Lemma \ref{cij(x,x)}, $f(\bO)=\bO$, and then $\bO\in V$.
  
Obviously $(x,y)\in \Delta $ if and only if $x_1-y_1=x_2-y_2=0$. We have $\Delta
\subset W$, $V\setminus \Delta \subset W\setminus \Delta$. If $(x,y)\in
W\setminus \Delta$ then by (\ref{wzKramer}) all $W_{ij}(x,y)=0$, 
so 
\begin{equation} \label{ab} V\setminus \Delta = W\setminus \Delta .\end{equation}

%************  LEMMA  u_imm   **************************************
\begin{lemma} \label{u_imm} $V\cap \Delta=\{ \bO\}$ as germs of sets if and only if $u$ has an isolated
 critical point at the origin.
\end{lemma}
\begin{proof}  Let $Du(x)$ denote the derivative matrix of $u$ at $x$. By Lemma
\ref{cij(x,x)}, for $(x,x)\in \Delta$ the determinant $W_{ij}(x,x)$ equals some
$2\times 2$--minor of the derivative matrix $Du(x)$. As $\Delta \subset W$, the
origin is an isolated critical point of $u$ if and only if it is an isolated
point of  $V\cap \Delta$. \end{proof} 

\bigskip

By the previous lemma and (\ref{ab}) we get
\begin{cor} \label{int_set}
If $u$ has an isolated critical point at the origin then $V\setminus
\{\bO\}=f\inv (\bO)\setminus \{\bO\}$ is the germ of
the set of self--intersection points  of $u$, and $D^2(u)=\{u(p)\ |\ \ (p,q)\in V\}$. 
\end{cor}

We shall say that a self--intersection  $(p,q)$, where $u(p)=u(q)$, is \emph{transverse}, if
\[ Du(p)\K^2+Du(q)\K^2=\K^3. \]
Since
\begin{align*}Dw(x,y)&= 
\left [
\begin{matrix}
\frac{\partial u_1}{\partial x_1}(x) & \frac{\partial u_1}{\partial x_2}(x) & -\frac{\partial u_1}{\partial x_1}(y) &
-\frac{\partial u_1}{\partial x_2}(y) \\
\frac{\partial u_2}{\partial x_1}(x) & \frac{\partial u_2}{\partial x_2}(x) & -\frac{\partial u_2}{\partial x_1}(y) &
-\frac{\partial u_2}{\partial x_2}(y) \\
\frac{\partial u_3}{\partial x_1}(x) & \frac{\partial u_3}{\partial x_2}(x) & -\frac{\partial u_3}{\partial x_1}(y) &
-\frac{\partial u_3}{\partial x_2}(y)
\end{matrix} \right ] \\
&=
\left [ Du(x)\ |\ -Du(y) \right ],
\end{align*}
 a self--intersection $(p,q)$ in $V\setminus \{ \bO\}$ is transverse if and only if $\rank Dw(p,q)=3$.

Let $\Opis_{K,\bO}$ denote the ring of germs of analytic functions at $\bO\in\K^4=\K^2\times\K^2$.

\begin{cor} \label{transv_inter} The germ $u$ has only transverse
self--intersections if and only if $\rank Dw(p,q)=3$  for each $(p,q)\in V\setminus\{\bO\}$.

In particular, if the ideal in $\Opis_{K,\bO}$ generated by $f_1,\ldots,f_6$ and all
$3\times 3$--minors of $Dw$ has an isolated zero at the origin, then $u$
has only transverse self--intersections near the origin.
\end{cor}

\bigskip

\begin{prop} \label{curve}
If the  germ $u:\K^2,\bO\rightarrow\K^3,\bO$ has an isolated critical point at $\bO$ and has only
transverse self-intersections, then $V=V(f_1,\ldots,f_6)\subset\K^4$ is a curve having
an isolated singularity at the origin, i.e. the origin is isolated in 
$\{(p,q)\in V\ |\ \rank Df(p,q)<3\}$.

If that is the case and $\K=\R$ then $V$ is an union of a finite collection of half--branches.
\end{prop}
\begin{proof}  By Corollary \ref{int_set}, $V$ is the set of self--intersections of $u$.

In some neighbourhood of the origin the rank of the matrix $Dw(p,q)$ equals $3$ at self--intersection points $(p,q)$, so the rank of the matrix $Df(p,q)$ is greater or equal to $3$. Hence $V$ is a curve having an isolated singularity at the origin. \end{proof}

\bigskip

Now we shall show how to check that $u$ has no triple points, i.e. such points $(p,q,r)$ that $p\neq q \neq r$ and $u(p)=u(q)=u(r)$.

\bigskip

Let us define 
$s\colon \K^6 \ar \K^6$, $s(x,y,z)=(w(x,y),w(y,z))$,
where $(x,y,z)=(x_1,x_2,y_1,y_2,z_1,z_2)\in \K ^6$, and let
\[\Sigma=\{ (x,y,z)\in \K ^6 \ |\ x=y\vee x=z \vee y=z \}.\]
Then $s\inv (\bO)\setminus \Sigma$ consists of such points $(p,q,r)$  that $u(p)=u(q)=u(r)$ is a  triple point of $u$. For $i=1,2,3$ we have
\begin{align*}
w_i(x,y)&= c_{i1}(x,y)(x_1-y_1)+
c_{i2}(x,y)(x_2-y_2)\\
w_i(y,z)&= c_{i1}(y,z)(y_1-z_1)+
c_{i2}(y,z)(y_2-z_2)\\
w_i(x,z)&=c_{i1}(x,z)(x_1-z_1)+
c_{i2}(x,z)(x_2-z_2).
\end{align*}
Define
\[ \tilde{V}=\{ (x,y,z)\in s\inv (\bO)\ |\ (x,y)\in V \wedge (y,z)\in V \wedge (x,z)\in V\}.\]
Applying the same arguments as above we obtain
\begin{equation} \label{ab2}
s\inv (\bO)\setminus \Sigma = \tilde{V}\setminus \Sigma\ .
\end{equation}

%**************** LEMMA tp_set **********************************
\begin{lemma} \label{tp_set} If $u$ has an isolated critical point at the origin
 then $\tilde{V}\cap \Sigma=\{ \bO\}$ as germs of sets.
\end{lemma}
\begin{proof}  By Lemma \ref{cij(x,x)}, the determinants $W_{ij}(x,x)$ are the 
$2\times 2$--minors of the derivative matrix $Du(x)$. 
For each $x\neq \bO$ close to $\bO$ there exists 
 $W_{ij}(x,x)\neq 0$. Then for each $t\in \K^2$ neither $(x,x,t)$ nor $(x,t,x)$ nor $(t,x,x)$
belongs to $\tilde{V}$. If points $(0,0,t)$, $(0,t,0)$, $(t,0,0)$ belong to
$\tilde{V}$ then $t\in u\inv (\bO)$, but $\bO$ is isolated in the zeroes
set of $u$. So we obtain the equality of germs $\tilde{V}\cap \Sigma=\{ \bO\}$. \end{proof} 

\bigskip

Let 
\[
\begin{array}{lll}
			S_1(x,y,z)=w_1(x,y) & S_2(x,y,z)=w_2(x,y) & S_3(x,y,z)=w_3(x,y) \\
			S_4(x,y,z)=w_1(y,z) & S_5(x,y,z)=w_2(y,z) & S_6(x,y,z)=w_3(y,z) \\
			S_7(x,y,z)=W_{12}(x,y) & S_8(x,y,z)=W_{13}(x,y) & S_9(x,y,z)=W_{23}(x,y) \\
			S_{10}(x,y,z)=W_{12}(y,z) & S_{11}(x,y,z)=W_{13}(y,z) & S_{12}(x,y,z)=W_{23}(y,z) \\
			S_{13}(x,y,z)=W_{12}(x,z) & S_{14}(x,y,z)=W_{13}(x,z) & S_{15}(x,y,z)=W_{23}(x,z). 
\end{array}
\]
If $u$ has an isolated critical point at the origin then $\tilde{V}\setminus \Sigma = \tilde{V}\setminus \{\bO\}=
V(S_1,\ldots,S_{15})$ is the set of triple points.

\medskip

Let us denote by $\K\{x,y,z\}$ the ring of convergent power series in variables $x_1,x_2,y_1,y_2,z_1,z_2$.
We get

\begin{cor} \label{triple}
If $u$ has an isolated singular point at the origin and
\[\dim_K \K\{x,y,z\}/\langle S_1,\ldots,S_{15} \rangle<+\infty,\]
then $u$ has no  triple points.

Hence, if the dimension is infinite then $u$ may have triple points. 
\end{cor}

\begin{theorem}\label{doublenumber}
Assume that an analytic germ $u:\R^2,\bO\rightarrow\R^3,\bO$ has an isolated critical point at $\bO$,
all self--intersections are transverse, and there are no triple points.

Then $V=V(w_1,w_2,w_3,W_{12},W_{13},W_{23})$ is a curve having an isolated singular point at the origin,
so that $V$
is an union of a finite collection of half--branches.
Each half--branch in the set of double points $D^2(u)$ is represented by two half--branches in $V$. 
\end{theorem}

Now we shall present three examples. In all the examples:
\begin{itemize}
\item $f=(f_1,\ldots ,f_6)=(w_1,w_2,w_3,W_{12},W_{13},W_{23})\colon (\R ^4,\bO)\ar (\R ^6,\bO)$
\item $I_2$ is the ideal generated by $f_i$'s and all $3\times 3$--minors of $D(f_1,\ldots ,f_6)$
\item $J=\left<f_1,\ldots ,f_6\right>$
\item $I_3$ is the ideal generated by $g_i$'s and all $3\times 3$--minors of $D(g_1,g_2,g_3)$
\item $J_1=\left<g_1,g_2,g_3,h\right>$
\item $J_2=\left<g_1,g_2,g_3,h^2\right>$
\item $H_+$ and $H_-$ are as in the Theorem \ref{nr18}.
\end{itemize}
\begin{ex}
Let $u=(x_1,x_1 x_2+x_2^3,x_1x_2^2+\frac{9}{10}x_2^4)$. 
Then $u$ has an isolated critical point at the origin.
First let us compute the germs $c_{ij}$:

\noindent
$c_{11}=1$\\ 
$c_{12}=0$\\ 
$c_{21}=x_2$\\ 
$c_{22}=y_1+x_2^2+x_2y_2+y_2^2$\\ 
$c_{31}=x_2^2$\\ $c_{32}=x_2y_1+y_1y_2+\frac{9}{10}x_2^3+
\frac{9}{10}x_2^2y_2+\frac{9}{10}x_2y_2^2+\frac{9}{10}y_2^3$

\noindent
and the determinants $W_{ij}$:

\noindent
$W_{12}=y_1+x_2^2+x_2y_2+y_2^2$\\ 
$W_{13}=x_2y_1+y_1y_2+\frac{9}{10}x_2^3+\frac{9}{10}x_2^2y_2+
\frac{9}{10}x_2y_2^2+\frac{9}{10}y_2^3$\\ $W_{23}=x_2y_1y_2-
\frac{1}{10}x_2^4-\frac{1}{10}x_2^3y_2-\frac{1}{10}x_2^2y_2^2+\frac{9}{10}x_2y_2^3$

\medskip

We have
 $\Bpis (I_2)=\{ (1,0,0,0),(0,0,1,0),(0,2,0,0),(0,1,0,1),(0,0,0,3)\}$.
By Proposition \ref{nr13}, $V_C(f_1,\ldots,f_6)$ is a curve having an isolated
singular point at the origin.
We have $\Bpis (J)=\{ (1,0,0,0),(0,0,1,0),(0,3,0,0)\}$.

Take $g_1=f_1+f_6$, $g_2=f_2+f_5$, $g_3=f_3+f_4$. 
We have\\ $\Bpis (I_3)=\{ (1,0,0,0),(0,0,1,0),(0,2,0,0),(0,1,0,1),(0,0,0,3)\}$,
so that the curve $V_C(g_1,g_2,g_3)$ has an isolated singularity at the origin. 

Take $h=f_1$. We may compute
$\Bpis (J_1)=\{ (1,0,0,0),(0,0,1,0),(0,3,0,0) \}$,
$\Bpis (J_2)=\{ (1,0,0,0),(0,0,1,0),(0,3,0,0) \}$. 
We have $J_1\subset J$ and $\Npis (J)\setminus \Npis (J_1)=\emptyset$,
 so that $\dim _R (J/J_1)=0$. By Lemma \ref{nr16}, $V_C(f_1,\ldots ,f_6)=V_C(g_1,g_2,g_3,h)$. 
As $\Bpis (J_1)=\Bpis (J_2)$, then $\xi=1$ and $k=2>1$.
We compute
\[ \deg _0(H_+)=3, \quad \deg _0(H_-)=-3.\]

Applying Proposition \ref{transv_inter} and Corollary \ref{triple}
we may check that all the self--intersections are transverse
and $u$ has no triple points.

According to Theorems \ref{nr18}, \ref{doublenumber},
there are $3$ half-branches in the set of double points $D^2(u)$.
\end{ex}

\begin{ex}
Let $u=(x_1^2-2x_2^2,x_1x_2+x_1^3,x_1x_2-x_2^3)$. Then $u$ has an isolated critical point at the origin.
We have\\ 
$\Bpis (I_2)=\{ (2,0,0,0),(1,1,0,0),(0,2,0,0),(1,0,2,0),(0,1,2,0),(0,0,3,0),$\\ $(1,0,1,2),(0,1,1,2),(0,0,2,2),(1,0,0,3),(0,1,0,3),(0,0,1,3),(0,0,0,5)\}$.
By Proposition \ref{nr13}, $V_C(f_1,\ldots,f_6)$ is a curve having an isolated singularity at
the origin.
We have\\
$\Bpis (J)=\{ (2,0,0,0),(1,1,0,0),(0,2,0,0),(1,0,2,0),(0,1,2,0),(0,0,3,0)\}$.

Take $g_1=f_1-3f_2+f_6$, $g_2=f_2-2f_5$, $g_3=f_3+3f_4+f_6$. We may compute\\
$\Bpis (I_3)=\{ (2,0,0,0),(1,1,0,0),(0,2,0,0),(1,0,3,0),(0,1,3,0),(0,0,4,0),$\\ $(1,0,2,1),(0,1,2,1),(0,0,3,1),(1,0,1,3),(0,1,1,3),(0,0,2,3),(1,0,0,5),$\\ $(0,1,0,5),(0,0,1,5),(0,0,0,6)\}$,
so that $V_C(g_1,g_2,g_3)$ has an isolated singularity at the origin. 

Take $h=f_1$. We may compute\\ 
$\Bpis (J_1)=\{ (2,0,0,0),(1,1,0,0),(0,2,0,0),(1,0,2,0),(0,1,3,0),(0,0,4,0),$\\ $(0,1,2,1),(0,0,3,1) \}$,\\
$\Bpis (J_2)=\{(2,0,0,0),(1,1,0,0),(0,2,0,0),(1,0,3,0),(0,1,3,0),(0,0,5,0),$\\ $(0,0,4,2),(1,0,2,4),(0,1,2,4),(0,0,3,4) \}$.

We have $J_1\subset J$ and the set $\Npis (J)\setminus \Npis (J_1)$ is finite, so that $\dim _R (J/J_1)<\infty$.
 Then $V_C(f_1,\ldots ,f_6)=V_C(g_1,g_2,g_3,h)$ (see Lemma 3.3). 
We have\\ $\max _{\beta \in \Npis (J_1)\setminus \Npis (J_2)}\{ |\beta|\}=6$,
 $\min _{\alpha \in \Bpis (J_1)\setminus \Npis (J_2)}\{ |\alpha|\}=3$, and we take $\xi=1+6-3=4$ and $k=6>4$.
We compute
\[ \deg _0(H_+)=1, \quad \deg _0(H_-)=-1.\]

Applying Proposition \ref{transv_inter} and Corollary \ref{triple}
we may check that all the self--intersections are transverse
and $u$ has no triple points.

According to Theorems \ref{nr18}, \ref{doublenumber}, there is one half-branche
in the set of double points $D^2(u)$.
\end{ex}

\begin{ex}
Let $u=(x_1,x_1x_2+x_2^3,x_1x_2^2+x_2^4)$. Then $u$ has an isolated critical point at the origin.

Applying Proposition \ref{transv_inter} and Corollary \ref{triple}
we may check that all the self--intersections are transverse,
and the germ $u$ may have triple points.

We have
$\Bpis (I_2)=\{ (1,0,0,0),(0,0,1,0),(0,2,0,0),(0,1,0,1),(0,0,0,3)\}$.
By Proposition \ref{nr13}, $V_C(f_1,\ldots, f_6)$ is a curve having an isolated
singularity at the origin.
We may compute 
$\Bpis (J)=\{ (1,0,0,0),(0,0,1,0),(0,2,0,1)\}$.

Take $g_1=f_1+f_6$, $g_2=f_2+f_5$, $g_3=f_3+f_4$, $h=f_1$.
As above, we may check that
$V_C(g_1,g_2,g_3)$ has an isolated singularity at the origin. 

Take $h=f_1$. We may compute
$\Bpis (J_1)=\{ (1,0,0,0),(0,0,1,0),(0,2,0,1) \}$, 
$\Bpis (J_2)=\{ (1,0,0,0),(0,0,1,0),(0,2,0,1)\}$.
We have $J_1\subset J$ and $\Npis (J)\setminus \Npis (J_1)=\emptyset$,
 so that $\dim _R (J/J_1)=0$. By Lemma \ref{nr16}, $V_C(f_1,\ldots ,f_6)=V_C(g_1,g_2,g_3,h)$.
 As 
$\Bpis (J_1)=\Bpis (J_1)$,
then $\xi=1$ and $k=2>1$.
We compute
\[ \deg _0(H_+)=3, \quad \deg _0(H_-)=-3.\]
According to Theorem \ref{nr18}, there are $6$ half-branches in the set 
of self--intersection points of $u$.
As $u$ may have triple points, we may only conclude that there are at most 3
half--branches in $D^2(u)$. In fact, the image $u(\R^2)$ contains  one half--branch
consisting of triple points
(see \cite{mararnunoballesteros}).
\end{ex}

\end{document}